\documentclass{amsart}

\usepackage[lite]{amsrefs}

\usepackage{amssymb}

\usepackage[all,cmtip]{xy}

\usepackage{tikz}

\usepackage{babel}

\usepackage{hyperref}
\usepackage{cleveref}

\usepackage{todonotes}

\DeclareMathOperator{\Id}{Id}
\DeclareMathOperator{\SL}{SL}

\numberwithin{equation}{section}

\theoremstyle{plain} 
\newtheorem{thm}[equation]{Theorem}
\newtheorem{con}[equation]{Conjecture}

\newtheorem{lem}[equation]{Lemma}

\theoremstyle{definition}
\newtheorem{defn}[equation]{Definition}

\theoremstyle{remark}

\begin{document}

\title{Index of the Kontsevich-Zorich monodromy of origamis in $\mathcal{H}(2)$}

\author{Pascal Kattler}

\address{}





\maketitle


\begin{abstract}
The Kontsevich-Zorich monodromy of an origami is the image of the action of the Veech group on the non tautological part of the homology. In this paper we make some progress to show, that for origamis in the stratum $\mathcal{H}(2)$ the index of the Kontsevich-Zorich monodromy in $SL_2(\mathbb{Z})$ is either 1 or 3.
\end{abstract}

\section{Introduction}
In this article we show most parts of a conjecture from \cite{bonnafoux2022arithmeticity}, namely that the index of the Kontsevich-Zorich monodromy of  primitive  origamis of degree $d$ in the stratum $\mathcal{H}(2)$ is either 1 or 3 in $\SL_2(\mathbb{Z})$. Hubert and Lelièvre showed in \cite{MR2214127}, that there are two $\SL_2(\mathbb{Z})$ orbits $\mathcal{A}_d$ and $\mathcal{B}_d$ if the degree $d$ is odd and one $\SL_2(\mathbb{Z})$-orbit, if the degree is even, distinguished by their HLK-invariant. Furthermore each orbit has as representative an \textbf{L-origamis} $L(n,m)$. This is an L-shaped origami as in Figure \labelcref{ZylinderDecom2}, where opposite edges are glued, $n$ is the number of squares in horizontal direction and $m$ in vertical direction. In the even case a representative is $L(n, m)$, where $n$ is even and $m$ is odd (or reserved). In the odd case representatives are $L(n,m)$, where $m$ and $n$ are even for $\mathcal{A}_d$ and $L(n,m)$, where $m$ and $n$ are odd for $\mathcal{B}_d$. So we will show the following Theorem.
\begin{thm}\label{main}
Let $\mathcal{O}$ be an origami of degree $d$ and genus 2 and $\Gamma \subseteq \SL_2(\mathbb{Z})$ the Kontsevich-Zorich monodromy of $\mathcal{O}$.
\begin{enumerate}
\item The index of $\Gamma$ in $\SL_2(\mathbb{Z})$ is at most 3, if $d$ is even. 
\item The index of $\Gamma$ in $\SL_2(\mathbb{Z})$ is 1, if $d$ is odd and $\mathcal{O}$ lies in $\mathcal{A}_d$. 
\end{enumerate}
\end{thm}
We proceed as follows. We choose for each degree and for each orbit an L-origami $\mathcal{O}$ as representative. Then we take two directions and their corresponding Dehn multitwists, which are elements of the Veech group of $\mathcal{O}$, as in \cite{MR1005006} (Proposition 2.4). Finally we compute the actions of this Dehn multitwists on the non tautological part of the homology and show that the indices of the groups, generated by them, is 1 or 3. 

The following statements still miss to complete the proof of the entire conjecture: In the index 3 case, we just showed, that the index is at most 3. Also the in the even case, the following conjecture still misses.
\begin{con}
In the setting of \Cref{main}, the index of $\Gamma$ in $\SL_2(\mathbb{Z})$ is 3, if $d$ is odd and $\mathcal{O}$ lies in $\mathcal{B}_d$.
\end{con}
\subsection*{Acknowledgments}
I am grateful for the  support provided by my supervisor Gabriela Weitze-Schmithüsen  throughout working on this paper. This work was founded by the Project-ID 286237555—TRR 195—by the Deutsche Forschungsgemeinschaft (DFG, German Research
Foundation).
\section{Computations of the indices}
\subsection{The odd case and the orbit $\mathcal{A}_d$}
We first show that the index of the Kontsevich-Zorich monodromy of the L-origami $L(2, 2n)$ for $ n \in \mathbb{N}$ (see picture below) is $1.$ 
In fact we show that the Kontsevich-Zorich monodromy is generated by the action of the  Dehn multitwists of the cylinders in directions $(0,1)$ and $(n , n+1)$ on the homology. We show, that this action is given by the matrices $\begin{pmatrix}
 1 & -1 \\  0 & 1
\end{pmatrix}$ and $ \begin{pmatrix}
2& -1 \\ 1 & 0
\end{pmatrix}$, with respect to a given basis of the homology. These matrices generate $\SL_2(\mathbb{Z}).$ The cylinder decomposition of the direction $(n, n+1)$ is as shown  in  Figure \labelcref{ZylinderDecom2} and Figure \labelcref{ZylinderDecom3}:
\begin{figure}
\begin{tikzpicture}[scale = 1]
\draw ( 0 , 0) -- ( 1 , 0) -- ( 1 , 1) -- ( 0 , 1) -- ( 0 , 0) ; 
\draw ( 1 , 0) -- ( 2 , 0) -- ( 2 , 1) -- ( 1 , 1) -- ( 1 , 0) ; 
\draw ( 0 , 1) -- ( 1 , 1) -- ( 1 , 2) -- ( 0 , 2) -- ( 0 , 1) ; 
\draw ( 0 , 2) -- ( 1 , 2) -- ( 1 , 3) -- ( 0 , 3) -- ( 0 , 2) ; 
\draw ( 0 , 3) -- ( 1 , 3) -- ( 1 , 4) -- ( 0 , 4) -- ( 0 , 3) ; 
\draw[color = red] ( 0. , 0. ) -- ( 0.66666666666666663 , 1.); 
 \draw[color = red] ( 0.66666666666666663 , 1. ) -- ( 1. , 1.5); 
 \draw[color = red] ( 0. , 1.5 ) -- ( 0.33333333333333331 , 2.); 
 \draw[color = red] ( 0.33333333333333331 , 2. ) -- ( 1. , 3.); 
 \draw[color = red] ( 0. , 3. ) -- ( 0.66666666666666663 , 4.); 
 \draw[color = red] ( 0.66666666666666663 , 0. ) -- ( 1. , 0.5); 
 \draw[color = red] ( 1. , 0.5 ) -- ( 1.3333333333333333 , 1.); 
 \draw[color = red] ( 1.3333333333333333 , 0. ) -- ( 2. , 1.); 
 \draw[color = red] ( 0. , 1. ) -- ( 0.66666666666666663 , 2.); 
 \draw[color = red] ( 0.66666666666666663 , 2. ) -- ( 1. , 2.5); 
 \draw[color = red] ( 0. , 2.5 ) -- ( 0.33333333333333331 , 3.); 
 \draw[color = red] ( 0.33333333333333331 , 3. ) -- ( 1. , 4.); 
 \draw[color = green] ( 1. , 0. ) -- ( 1.6666666666666667 , 1.); 
 \draw[color = green] ( 1.6666666666666667 , 0. ) -- ( 2. , 0.5); 
 \draw[color = green] ( 0. , 0.5 ) -- ( 0.33333333333333331 , 1.); 
 \draw[color = green] ( 0.33333333333333331 , 1. ) -- ( 1. , 2.); 
 \draw[color = green] ( 0. , 2. ) -- ( 0.66666666666666663 , 3.); 
 \draw[color = green] ( 0.66666666666666663 , 3. ) -- ( 1. , 3.5); 
 \draw[color = green] ( 0. , 3.5 ) -- ( 0.33333333333333331 , 4.); 
 \draw[color = green] ( 0.33333333333333331 , 0. ) -- ( 1. , 1.); 
\end{tikzpicture}
\caption{Cylinder decomposition of $L(2,2n)$ with $n = 2$ in direction $(2,3)$}
\label{ZylinderDecom2}
\end{figure}
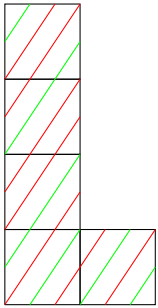
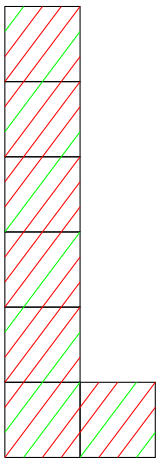
\begin{figure}
\begin{tikzpicture}[scale = 1]
\draw ( 0 , 0) -- ( 1 , 0) -- ( 1 , 1) -- ( 0 , 1) -- ( 0 , 0) ; 
\draw ( 1 , 0) -- ( 2 , 0) -- ( 2 , 1) -- ( 1 , 1) -- ( 1 , 0) ; 
\draw ( 0 , 1) -- ( 1 , 1) -- ( 1 , 2) -- ( 0 , 2) -- ( 0 , 1) ; 
\draw ( 0 , 2) -- ( 1 , 2) -- ( 1 , 3) -- ( 0 , 3) -- ( 0 , 2) ; 
\draw ( 0 , 3) -- ( 1 , 3) -- ( 1 , 4) -- ( 0 , 4) -- ( 0 , 3) ; 
\draw ( 0 , 4) -- ( 1 , 4) -- ( 1 , 5) -- ( 0 , 5) -- ( 0 , 4) ; 
\draw ( 0 , 5) -- ( 1 , 5) -- ( 1 , 6) -- ( 0 , 6) -- ( 0 , 5) ; 
\draw[color = red] ( 0. , 0. ) -- ( 0.75 , 1.); 
 \draw[color = red] ( 0.75 , 1. ) -- ( 1. , 1.3333333333333333); 
 \draw[color = red] ( 0. , 1.3333333333333333 ) -- ( 0.5 , 2.); 
 \draw[color = red] ( 0.5 , 2. ) -- ( 1. , 2.6666666666666665); 
 \draw[color = red] ( 0. , 2.6666666666666665 ) -- ( 0.25 , 3.); 
 \draw[color = red] ( 0.25 , 3. ) -- ( 1. , 4.); 
 \draw[color = red] ( 0. , 4. ) -- ( 0.75 , 5.); 
 \draw[color = red] ( 0.75 , 5. ) -- ( 1. , 5.333333333333333); 
 \draw[color = red] ( 0. , 5.333333333333333 ) -- ( 0.5 , 6.); 
 \draw[color = red] ( 0.5 , 0. ) -- ( 1. , 0.66666666666666663); 
 \draw[color = red] ( 1. , 0.66666666666666663 ) -- ( 1.25 , 1.); 
 \draw[color = red] ( 1.25 , 0. ) -- ( 2. , 1.); 
 \draw[color = red] ( 0. , 1. ) -- ( 0.75 , 2.); 
 \draw[color = red] ( 0.75 , 2. ) -- ( 1. , 2.3333333333333335); 
 \draw[color = red] ( 0. , 2.3333333333333335 ) -- ( 0.5 , 3.); 
 \draw[color = red] ( 0.5 , 3. ) -- ( 1. , 3.6666666666666665); 
 \draw[color = red] ( 0. , 3.6666666666666665 ) -- ( 0.25 , 4.); 
 \draw[color = red] ( 0.25 , 4. ) -- ( 1. , 5.); 
 \draw[color = red] ( 0. , 5. ) -- ( 0.75 , 6.); 
 \draw[color = red] ( 0.75 , 0. ) -- ( 1. , 0.33333333333333331); 
 \draw[color = red] ( 1. , 0.33333333333333331 ) -- ( 1.5 , 1.); 
 \draw[color = red] ( 1.5 , 0. ) -- ( 2. , 0.66666666666666663); 
 \draw[color = red] ( 0. , 0.66666666666666663 ) -- ( 0.25 , 1.); 
 \draw[color = red] ( 0.25 , 1. ) -- ( 1. , 2.); 
 \draw[color = red] ( 0. , 2. ) -- ( 0.75 , 3.); 
 \draw[color = red] ( 0.75 , 3. ) -- ( 1. , 3.3333333333333335); 
 \draw[color = red] ( 0. , 3.3333333333333335 ) -- ( 0.5 , 4.); 
 \draw[color = red] ( 0.5 , 4. ) -- ( 1. , 4.666666666666667); 
 \draw[color = red] ( 0. , 4.666666666666667 ) -- ( 0.25 , 5.); 
 \draw[color = red] ( 0.25 , 5. ) -- ( 1. , 6.); 
 \draw[color = green] ( 1. , 0. ) -- ( 1.75 , 1.); 
 \draw[color = green] ( 1.75 , 0. ) -- ( 2. , 0.33333333333333331); 
 \draw[color = green] ( 0. , 0.33333333333333331 ) -- ( 0.5 , 1.); 
 \draw[color = green] ( 0.5 , 1. ) -- ( 1. , 1.6666666666666667); 
 \draw[color = green] ( 0. , 1.6666666666666667 ) -- ( 0.25 , 2.); 
 \draw[color = green] ( 0.25 , 2. ) -- ( 1. , 3.); 
 \draw[color = green] ( 0. , 3. ) -- ( 0.75 , 4.); 
 \draw[color = green] ( 0.75 , 4. ) -- ( 1. , 4.333333333333333); 
 \draw[color = green] ( 0. , 4.333333333333333 ) -- ( 0.5 , 5.); 
 \draw[color = green] ( 0.5 , 5. ) -- ( 1. , 5.666666666666667); 
 \draw[color = green] ( 0. , 5.666666666666667 ) -- ( 0.25 , 6.); 
 \draw[color = green] ( 0.25 , 0. ) -- ( 1. , 1.);
\end{tikzpicture}
\caption{Cylinder decomposition of $L(2,2n)$ with $n = 3$ in direction: $(3,4)$}
\label{ZylinderDecom3}
\end{figure} 
We get always three saddle connections in the direction $(n, n+1).$ The green saddle connection is the one,  starting in the left bottom corner of the rightmost square. The two red saddle connections $r_1$ and $r_2$ are the other ones.  See \Cref{comp} for the proofs of the statements presented in the following.
We call \begin{enumerate}
\item[(i)] $\Theta_g$ the green cylinder. That is the cylinder with the green saddle connection as the upper boundary.
\item[(ii)] $\Theta_r$ the red cylinder. That is the cylinder with the red saddle connections as the upper boundary. 
\item[(iii)] $Y_1$ is the small vertical cylinder.
\item[(iv)] $Y_2$ is the large vertical cylinder.
\end{enumerate}
Note that we use the name of the cylinders also as name for their mid curves, which we consider as elements of the homology.

Now we remind of the definition of the combinatorial length (or width).
\begin{defn}
The \textbf{combinatorial length} of a cylinder of an origami $\pi \colon \mathcal{O}\to E $ with mid curve $\gamma$ is the multiplicity of the curve $\pi \circ \gamma,$ i.e. $\#\{t \in (0,1], \pi(\gamma(t)) = \pi(\gamma(0))\}.$
\end{defn}
Note that the definition of the combinatorial length is equivalent to the definition of the combinatorial width from  \cite{bonnafoux2022arithmeticity}.
The former cylinders have the combinatorial length $f_{\Theta_r} = 2n-1, f_{\Theta_g} = 2, f_{Y_2} = 2n, f_{Y_1} = 1$ and the combinatorial heights (this are the $c_i$ defined below) of these cylinders is 1, since the cylinders $\Theta_g$ and $\Theta_r$ have the same height, the same holds for $Y_1$ and $Y_2.$
The cylinder decomposition in direction $(n, n+1)$ defines (up to inverse) a unique affine map, which is a minimal Dehn multitwists along the core curves of these cylinders. Let $D_\Theta$ be the Dehn multitwist in  direction $(n, n+1)$ and analogously let  $D_Y$ be the Dehn multitwist in direction $(0, 1)$.

We take as basis of the homology the horizontal and vertical mid curves $X_1, X_2, Y_1, Y_2$ as in Figure \labelcref{homology} and as basis of the non-tautological part 
\begin{align*}
	X &= X_2 - 2 X_1\\
	Y &= Y_2 - 2nY_1
\end{align*}

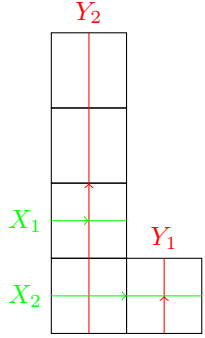
\begin{figure}
\begin{tikzpicture}
\draw ( 0 , 0) -- ( 1 , 0) -- ( 1 , 1) -- ( 0 , 1) -- ( 0 , 0) ; 
\draw ( 1 , 0) -- ( 2 , 0) -- ( 2 , 1) -- ( 1 , 1) -- ( 1 , 0) ; 
\draw ( 0 , 1) -- ( 1 , 1) -- ( 1 , 2) -- ( 0 , 2) -- ( 0 , 1) ; 
\draw ( 0 , 2) -- ( 1 , 2) -- ( 1 , 3) -- ( 0 , 3) -- ( 0 , 2) ; 
\draw ( 0 , 3) -- ( 1 , 3) -- ( 1 , 4) -- ( 0 , 4) -- ( 0 , 3) ; 
\draw[color = red, ->] (0.5,0) -- (0.5, 2);
\draw[color = red] (0.5,2) -- (0.5, 4) node[ above ]{$Y_2$};
\draw[color = red, ->] (1.5,0) -- (1.5, 0.5);
\draw[color = red] (1.5,0.5) -- (1.5, 1) node[ above ]{$Y_1$};
\draw[color = green, ->] (0, 0.5) node [left]{$X_2$} -- (1, 0.5);
\draw[color = green] (1,0.5) -- (2, 0.5);
\draw[color = green, ->] (0, 1.5) node [left]{$X_1$} -- (0.5, 1.5);
\draw[color = green] (0.5,1.5) -- (1, 1.5);
\end{tikzpicture}
\caption{the basis of the homology}
\label{homology}
\end{figure}

Let $D$ be the Dehn multitwist in a two cylinder direction with mid cylinders $\gamma_1$ and $\gamma_2.$
We compute the action $D_\ast$ of the Dehn multitwist $D$ in a  on the non-tautological part of the homology with the formula
\[
	D_\ast = \Id + c_1 f_2 \Omega(\cdot, \gamma_1 )\gamma_1 + c_2 f_1 \Omega(\cdot, \gamma_2 )\gamma_2
\]
from \cite{bonnafoux2022arithmeticity} (Chapter 2.4). The $f_i$ are the combinatorial lengths of $\gamma_i$ and $c_i$ are the smallest integers, such that the $c_1 h(\gamma_2) = c_2 h(\gamma_1))$, where $h(\gamma_i)$ is the height of $\gamma_i.$

In Figure \labelcref{intersectionNumberEven} we summarize the needed intersection numbers. Note that the sign of the intersection number $\Omega(s,t)$ of curves $s$ and $t$ in direction $v$ and $w$ is the sign of the determinant of the matrix $\begin{pmatrix}
v & w 
\end{pmatrix}$. 
We compute the not obvious intersection numbers in \Cref{comp}. 

Now we compute the cylinder middles $\Theta_r$ and $\Theta_g$ as linear combination of our chosen basis $(X_1, X_2, Y_1, Y_2)$ of the homology:
 \begin{align*}
 \Theta_r &= a X_1 + bX_2 + cY_1 + dY_2
 \end{align*}
 Let 
 \begin{align*}
 A =\begin{pmatrix}
 0 & 0 & 0 & 1 \\
 0 & 0 & 1 & 1 \\ 
 0 & -1 & 0 & 0 \\ 
 -1 & -1 & 0 & 0 \\ 
 \end{pmatrix}
 \end{align*}
 be the fundamental matrix of the intersection form with respect to the basis of the homology. Because of the bilinearity, we get $x =(a,b,c,d)$ is the unique solution of the equation $Ax = (\Omega(\Theta_r, X_1), \Omega(\Theta_r, X_2), \Omega(\Theta_r, Y_1), \Omega(\Theta_r, Y_2))^t.$ The same holds for $\Theta_g$, hence we get
 \begin{align*}
 	\Theta_r &=(n-1)X_2 +((2n-3)n +2)X_1 + nY_2 +(n-1)Y_1 \text{ and}\\
 	\Theta_g &= X_2 + 2(n-1)X_1 +Y_2 +2Y_1.
 \end{align*}
It follows
\begin{align*}
D_\Theta(X) &= X + c_{\Theta_g}f_{\Theta_r}\Omega(X, \Theta_g)\Theta_g
				+ c_{\Theta_r}f_{\Theta_g}\Omega(X, \Theta_r)\Theta_r\\
			&= X + (2n-1)\Theta_g - 2\Theta_r\\
			&= X + (2n-1)X_2 + 2(2n-1)(n-1)X_1 + (2n-1)Y_2 + (2n-1)2Y_1 \\
			& \quad   -2(n-1)X_2 -2((2n-3)n +2)X_1 - 2nY_2 - 2(n-1)Y_1\\
			&= X + X_2 - 2X_1 + 2nY_1 -Y_2 = 2X -Y\\
D_\Theta(Y) &= Y + c_{\Theta_g}f_{\Theta_r}\Omega(Y, \Theta_g)\Theta_g
				+ c_{\Theta_r}f_{\Theta_g}\Omega(Y, \Theta_r)\Theta_r\\
			&= Y - (n-1)\Theta_g + 2\Theta_r\\
			&= Y + (2n-1)X_2 + 2(2n-1)^2X_1 + (2n-1)Y_2 + (2n-1)2Y_1 \\
			& \quad   -2(n-1)X_2 -2((2n-3)n +2)X_1 - 2nY_2 - 2(n-1)Y_1\\
			&= Y + X_2 - 2X_1 + 2nY_1 - Y_2 = X
\end{align*}

and
\begin{align*}
	D_Y( X ) &= X + c_{Y_2}f_{Y_1}\Omega(X, Y_2)Y_2
				+ c_{Y_1}f_{Y_2}\Omega(X, Y_1)Y_1\\
			&= X - Y_2 + 2nY_1 = X-Y\\
	D_Y(Y) &= Y			
\end{align*}
So we have 
\begin{align*}
D_\Theta = \begin{pmatrix}
2 & 1\\-1 & 0
\end{pmatrix}
\text{  and  }
D_Y =
\begin{pmatrix}
1 & 0 \\ -1 & 1
\end{pmatrix}
\end{align*}

\begin{figure}
\centering
\begin{minipage}{.5\linewidth}
\begin{tabular}{c | c | c } 
$\Omega(X_i, \Theta_j)$ & $\Theta_r$ & $ \Theta_g$\\
\hline
$X_2$ & $2n-1$ &  $3$\\
\hline
$X_1$ & $n$ &  $1$\\
\hline
$X$ & $-1$ &  $1$\\
\end{tabular}

\begin{tabular}{c | c | c } 
$\Omega(X_i, Y_j)$ & $Y_1$ & $ Y_2$\\
\hline
$X_2$ & $1$ &  $1$\\
\hline
$X_1$ & $0$ &  $1$\\
\hline
$X$ & $1$ &  $-1$\\
\end{tabular}
\end{minipage}
\qquad
\begin{minipage}{.5\linewidth}
\begin{tabular}{c | c | c } 
$\Omega(Y_i, \Theta_j)$ & $\Theta_r$ & $ \Theta_g$\\
\hline
$Y_1$ & $-(n-1)$ &  $-1$\\
\hline
$Y_2$ & $-(2n-2)n - 1$ &  $-(2n-1)$\\
\hline
$Y$ & $-1$ &  $1$\\
\end{tabular}

\begin{tabular}{c | c | c } 
$\Omega(Y_i, Y_j)$ & $Y_1$ & $Y_2$\\
\hline
$Y_1$ & $0$ &  $0$\\
\hline
$Y_2$ & $0$ &  $0$\\
\hline
$Y$ & $0$ &  $0$\\
\end{tabular}
\end{minipage}
\caption{Intersection numbers of cylinder middles and the homology}
\label{intersectionNumberEven}
\end{figure}

\subsection{The even case}
In this chapter, we treat the origamis in  $\mathcal{H}(2)$ of even degree. In this case  we choose as representatives $L(2, 2n+1), n \in \mathbb{N}$. We show that the index of the Kontsevich-Zorich monodromy in $\SL_2(\mathbb{Z})$ is at most three. In fact we show that the group generated by the action of the Dehn multitwists in directions $(2n+2, 2n+1)$ and $(2n+1, 2n+3)$ is the index 3 subgroup generated by the matrices $\begin{pmatrix}
3 & 2\\-2 & -1
\end{pmatrix} $
and 
$
\begin{pmatrix}
1 & 0 \\ -1 & 1
\end{pmatrix}
$.
We get three saddle connections in direction $(2n+1, 2n+3)$ (see Figure \labelcref{ZylinderdecomPsi1} and Figure \labelcref{ZylinderdecomPsi2}). The green saddle connection $g$ is that, starting in the left bottom corner of the rightmost square. The two red saddle connections $r_1$ and $r_2$ are the other ones.
And we get three saddle connections in direction $(2n+2, 2n+1)$ (see Figure \labelcref{ZylinderdecomTheta1} and Figure \labelcref{ZylinderdecomTheta2}). The blue saddle connection $b$ is that, starting in the left bottom corner of the second square from the bottom. The two magenta saddle connections $m_1$ and $m_2$ are  the other ones. 
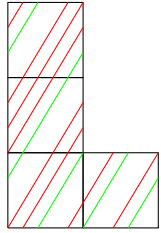
\begin{figure}
\begin{tikzpicture}[scale = 1]
\draw ( 0 , 0) -- ( 1 , 0) -- ( 1 , 1) -- ( 0 , 1) -- ( 0 , 0) ; 
\draw ( 1 , 0) -- ( 2 , 0) -- ( 2 , 1) -- ( 1 , 1) -- ( 1 , 0) ; 
\draw ( 0 , 1) -- ( 1 , 1) -- ( 1 , 2) -- ( 0 , 2) -- ( 0 , 1) ; 
\draw ( 0 , 2) -- ( 1 , 2) -- ( 1 , 3) -- ( 0 , 3) -- ( 0 , 2) ; 
\draw[color = red] ( 0. , 0. ) -- ( 0.59999999999999998 , 1.); 
 \draw[color = red] ( 0.59999999999999998 , 1. ) -- ( 1. , 1.6666666666666667); 
 \draw[color = red] ( 0. , 1.6666666666666667 ) -- ( 0.20000000000000001 , 2.); 
 \draw[color = red] ( 0.20000000000000001 , 2. ) -- ( 0.80000000000000004 , 3.); 
 \draw[color = red] ( 0.80000000000000004 , 0. ) -- ( 1. , 0.33333333333333331); 
 \draw[color = red] ( 1. , 0.33333333333333331 ) -- ( 1.3999999999999999 , 1.); 
 \draw[color = red] ( 1.3999999999999999 , 0. ) -- ( 2. , 1.); 
 \draw[color = red] ( 0. , 1. ) -- ( 0.59999999999999998 , 2.); 
 \draw[color = red] ( 0.59999999999999998 , 2. ) -- ( 1. , 2.6666666666666665); 
 \draw[color = red] ( 0. , 2.6666666666666665 ) -- ( 0.20000000000000001 , 3.); 
 \draw[color = red] ( 0.20000000000000001 , 0. ) -- ( 0.80000000000000004 , 1.); 
 \draw[color = red] ( 0.80000000000000004 , 1. ) -- ( 1. , 1.3333333333333333); 
 \draw[color = red] ( 0. , 1.3333333333333333 ) -- ( 0.40000000000000002 , 2.); 
 \draw[color = red] ( 0.40000000000000002 , 2. ) -- ( 1. , 3.); 
 \draw[color = green] ( 1. , 0. ) -- ( 1.6000000000000001 , 1.); 
 \draw[color = green] ( 1.6000000000000001 , 0. ) -- ( 2. , 0.66666666666666663); 
 \draw[color = green] ( 0. , 0.66666666666666663 ) -- ( 0.20000000000000001 , 1.); 
 \draw[color = green] ( 0.20000000000000001 , 1. ) -- ( 0.80000000000000004 , 2.); 
 \draw[color = green] ( 0.80000000000000004 , 2. ) -- ( 1. , 2.3333333333333335); 
 \draw[color = green] ( 0. , 2.3333333333333335 ) -- ( 0.40000000000000002 , 3.); 
 \draw[color = green] ( 0.40000000000000002 , 0. ) -- ( 1. , 1.); 
\end{tikzpicture}
\caption{Cylinder decomposition of $L(2,2n + 1)$ with $n = 1$ in direction (3,5)}
\label{ZylinderdecomPsi1}
\end{figure}

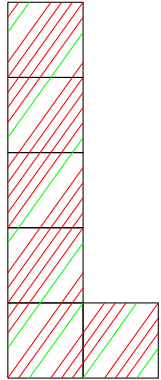
\begin{figure}
\begin{tikzpicture}[scale = 1]
\draw ( 0 , 0) -- ( 1 , 0) -- ( 1 , 1) -- ( 0 , 1) -- ( 0 , 0) ; 
\draw ( 1 , 0) -- ( 2 , 0) -- ( 2 , 1) -- ( 1 , 1) -- ( 1 , 0) ; 
\draw ( 0 , 1) -- ( 1 , 1) -- ( 1 , 2) -- ( 0 , 2) -- ( 0 , 1) ; 
\draw ( 0 , 2) -- ( 1 , 2) -- ( 1 , 3) -- ( 0 , 3) -- ( 0 , 2) ; 
\draw ( 0 , 3) -- ( 1 , 3) -- ( 1 , 4) -- ( 0 , 4) -- ( 0 , 3) ; 
\draw ( 0 , 4) -- ( 1 , 4) -- ( 1 , 5) -- ( 0 , 5) -- ( 0 , 4) ; 
\draw[color = red] ( 0. , 0. ) -- ( 0.7142857142857143 , 1.); 
 \draw[color = red] ( 0.7142857142857143 , 1. ) -- ( 1. , 1.3999999999999999); 
 \draw[color = red] ( 0. , 1.3999999999999999 ) -- ( 0.42857142857142855 , 2.); 
 \draw[color = red] ( 0.42857142857142855 , 2. ) -- ( 1. , 2.7999999999999998); 
 \draw[color = red] ( 0. , 2.7999999999999998 ) -- ( 0.14285714285714285 , 3.); 
 \draw[color = red] ( 0.14285714285714285 , 3. ) -- ( 0.8571428571428571 , 4.); 
 \draw[color = red] ( 0.8571428571428571 , 4. ) -- ( 1. , 4.2000000000000002); 
 \draw[color = red] ( 0. , 4.2000000000000002 ) -- ( 0.5714285714285714 , 5.); 
 \draw[color = red] ( 0.5714285714285714 , 0. ) -- ( 1. , 0.59999999999999998); 
 \draw[color = red] ( 1. , 0.59999999999999998 ) -- ( 1.2857142857142858 , 1.); 
 \draw[color = red] ( 1.2857142857142858 , 0. ) -- ( 2. , 1.); 
 \draw[color = red] ( 0. , 1. ) -- ( 0.7142857142857143 , 2.); 
 \draw[color = red] ( 0.7142857142857143 , 2. ) -- ( 1. , 2.3999999999999999); 
 \draw[color = red] ( 0. , 2.3999999999999999 ) -- ( 0.42857142857142855 , 3.); 
 \draw[color = red] ( 0.42857142857142855 , 3. ) -- ( 1. , 3.7999999999999998); 
 \draw[color = red] ( 0. , 3.7999999999999998 ) -- ( 0.14285714285714285 , 4.); 
 \draw[color = red] ( 0.14285714285714285 , 4. ) -- ( 0.8571428571428571 , 5.); 
 \draw[color = red] ( 0.8571428571428571 , 0. ) -- ( 1. , 0.20000000000000001); 
 \draw[color = red] ( 1. , 0.20000000000000001 ) -- ( 1.5714285714285714 , 1.); 
 \draw[color = red] ( 1.5714285714285714 , 0. ) -- ( 2. , 0.59999999999999998); 
 \draw[color = red] ( 0. , 0.59999999999999998 ) -- ( 0.2857142857142857 , 1.); 
 \draw[color = red] ( 0.2857142857142857 , 1. ) -- ( 1. , 2.); 
 \draw[color = red] ( 0. , 2. ) -- ( 0.7142857142857143 , 3.); 
 \draw[color = red] ( 0.7142857142857143 , 3. ) -- ( 1. , 3.3999999999999999); 
 \draw[color = red] ( 0. , 3.3999999999999999 ) -- ( 0.42857142857142855 , 4.); 
 \draw[color = red] ( 0.42857142857142855 , 4. ) -- ( 1. , 4.7999999999999998); 
 \draw[color = red] ( 0. , 4.7999999999999998 ) -- ( 0.14285714285714285 , 5.); 
 \draw[color = red] ( 0.14285714285714285 , 0. ) -- ( 0.8571428571428571 , 1.); 
 \draw[color = red] ( 0.8571428571428571 , 1. ) -- ( 1. , 1.2); 
 \draw[color = red] ( 0. , 1.2 ) -- ( 0.5714285714285714 , 2.); 
 \draw[color = red] ( 0.5714285714285714 , 2. ) -- ( 1. , 2.6000000000000001); 
 \draw[color = red] ( 0. , 2.6000000000000001 ) -- ( 0.2857142857142857 , 3.); 
 \draw[color = red] ( 0.2857142857142857 , 3. ) -- ( 1. , 4.); 
 \draw[color = red] ( 0. , 4. ) -- ( 0.7142857142857143 , 5.); 
 \draw[color = red] ( 0.7142857142857143 , 0. ) -- ( 1. , 0.40000000000000002); 
 \draw[color = red] ( 1. , 0.40000000000000002 ) -- ( 1.4285714285714286 , 1.); 
 \draw[color = red] ( 1.4285714285714286 , 0. ) -- ( 2. , 0.80000000000000004); 
 \draw[color = red] ( 0. , 0.80000000000000004 ) -- ( 0.14285714285714285 , 1.); 
 \draw[color = red] ( 0.14285714285714285 , 1. ) -- ( 0.8571428571428571 , 2.); 
 \draw[color = red] ( 0.8571428571428571 , 2. ) -- ( 1. , 2.2000000000000002); 
 \draw[color = red] ( 0. , 2.2000000000000002 ) -- ( 0.5714285714285714 , 3.); 
 \draw[color = red] ( 0.5714285714285714 , 3. ) -- ( 1. , 3.6000000000000001); 
 \draw[color = red] ( 0. , 3.6000000000000001 ) -- ( 0.2857142857142857 , 4.); 
 \draw[color = red] ( 0.2857142857142857 , 4. ) -- ( 1. , 5.); 
 \draw[color = green] ( 1. , 0. ) -- ( 1.7142857142857142 , 1.); 
 \draw[color = green] ( 1.7142857142857142 , 0. ) -- ( 2. , 0.40000000000000002); 
 \draw[color = green] ( 0. , 0.40000000000000002 ) -- ( 0.42857142857142855 , 1.); 
 \draw[color = green] ( 0.42857142857142855 , 1. ) -- ( 1. , 1.8); 
 \draw[color = green] ( 0. , 1.8 ) -- ( 0.14285714285714285 , 2.); 
 \draw[color = green] ( 0.14285714285714285 , 2. ) -- ( 0.8571428571428571 , 3.); 
 \draw[color = green] ( 0.8571428571428571 , 3. ) -- ( 1. , 3.2000000000000002); 
 \draw[color = green] ( 0. , 3.2000000000000002 ) -- ( 0.5714285714285714 , 4.); 
 \draw[color = green] ( 0.5714285714285714 , 4. ) -- ( 1. , 4.5999999999999996); 
 \draw[color = green] ( 0. , 4.5999999999999996 ) -- ( 0.2857142857142857 , 5.); 
 \draw[color = green] ( 0.2857142857142857 , 0. ) -- ( 1. , 1.); 
\end{tikzpicture}
\caption{Cylinder decomposition of $L(2,2n + 1)$ with $n = 2$ in direction (5,7)}
\label{ZylinderdecomPsi2}
\end{figure}

\begin{figure}
\begin{tikzpicture}[scale = 1]
\draw ( 0 , 0) -- ( 1 , 0) -- ( 1 , 1) -- ( 0 , 1) -- ( 0 , 0) ; 
\draw ( 1 , 0) -- ( 2 , 0) -- ( 2 , 1) -- ( 1 , 1) -- ( 1 , 0) ; 
\draw ( 0 , 1) -- ( 1 , 1) -- ( 1 , 2) -- ( 0 , 2) -- ( 0 , 1) ; 
\draw ( 0 , 2) -- ( 1 , 2) -- ( 1 , 3) -- ( 0 , 3) -- ( 0 , 2) ; 
\draw[color = magenta] ( 0. , 0. ) -- ( 1. , 0.75); 
 \draw[color = magenta] ( 1. , 0.75 ) -- ( 1.3333333333333333 , 1.); 
 \draw[color = magenta] ( 1.3333333333333333 , 0. ) -- ( 2. , 0.5); 
 \draw[color = magenta] ( 0. , 0.5 ) -- ( 0.66666666666666663 , 1.); 
 \draw[color = magenta] ( 0.66666666666666663 , 1. ) -- ( 1. , 1.25); 
 \draw[color = magenta] ( 0. , 1.25 ) -- ( 1. , 2.); 
 \draw[color = magenta] ( 0. , 2. ) -- ( 1. , 2.75); 
 \draw[color = magenta] ( 0. , 2.75 ) -- ( 0.33333333333333331 , 3.); 
 \draw[color = magenta] ( 0.33333333333333331 , 0. ) -- ( 1. , 0.5); 
 \draw[color = magenta] ( 1. , 0.5 ) -- ( 1.6666666666666667 , 1.); 
 \draw[color = magenta] ( 1.6666666666666667 , 0. ) -- ( 2. , 0.25); 
 \draw[color = magenta] ( 0. , 0.25 ) -- ( 1. , 1.); 
 \draw[color = magenta] ( 1. , 0. ) -- ( 2. , 0.75); 
 \draw[color = magenta] ( 0. , 0.75 ) -- ( 0.33333333333333331 , 1.); 
 \draw[color = magenta] ( 0.33333333333333331 , 1. ) -- ( 1. , 1.5); 
 \draw[color = magenta] ( 0. , 1.5 ) -- ( 0.66666666666666663 , 2.); 
 \draw[color = magenta] ( 0.66666666666666663 , 2. ) -- ( 1. , 2.25); 
 \draw[color = magenta] ( 0. , 2.25 ) -- ( 1. , 3.); 
 \draw[color = blue] ( 0. , 1. ) -- ( 1. , 1.75); 
 \draw[color = blue] ( 0. , 1.75 ) -- ( 0.33333333333333331 , 2.); 
 \draw[color = blue] ( 0.33333333333333331 , 2. ) -- ( 1. , 2.5); 
 \draw[color = blue] ( 0. , 2.5 ) -- ( 0.66666666666666663 , 3.); 
 \draw[color = blue] ( 0.66666666666666663 , 0. ) -- ( 1. , 0.25); 
 \draw[color = blue] ( 1. , 0.25 ) -- ( 2. , 1.);
\end{tikzpicture}
\caption{Cylinder decomposition of $L(2,2n+1)$ with $n = 1$ in direction $(4,3)$}
\label{ZylinderdecomTheta1}
\end{figure}
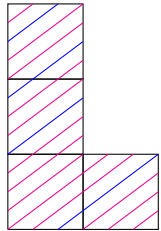

\begin{figure}
\begin{tikzpicture}[scale = 1]
\draw ( 0 , 0) -- ( 1 , 0) -- ( 1 , 1) -- ( 0 , 1) -- ( 0 , 0) ; 
\draw ( 1 , 0) -- ( 2 , 0) -- ( 2 , 1) -- ( 1 , 1) -- ( 1 , 0) ; 
\draw ( 0 , 1) -- ( 1 , 1) -- ( 1 , 2) -- ( 0 , 2) -- ( 0 , 1) ; 
\draw ( 0 , 2) -- ( 1 , 2) -- ( 1 , 3) -- ( 0 , 3) -- ( 0 , 2) ; 
\draw ( 0 , 3) -- ( 1 , 3) -- ( 1 , 4) -- ( 0 , 4) -- ( 0 , 3) ; 
\draw ( 0 , 4) -- ( 1 , 4) -- ( 1 , 5) -- ( 0 , 5) -- ( 0 , 4) ; 
\draw[color = magenta] ( 0. , 0. ) -- ( 1. , 0.83333333333333337); 
 \draw[color = magenta] ( 1. , 0.83333333333333337 ) -- ( 1.2 , 1.); 
 \draw[color = magenta] ( 1.2 , 0. ) -- ( 2. , 0.66666666666666663); 
 \draw[color = magenta] ( 0. , 0.66666666666666663 ) -- ( 0.40000000000000002 , 1.); 
 \draw[color = magenta] ( 0.40000000000000002 , 1. ) -- ( 1. , 1.5); 
 \draw[color = magenta] ( 0. , 1.5 ) -- ( 0.59999999999999998 , 2.); 
 \draw[color = magenta] ( 0.59999999999999998 , 2. ) -- ( 1. , 2.3333333333333335); 
 \draw[color = magenta] ( 0. , 2.3333333333333335 ) -- ( 0.80000000000000004 , 3.); 
 \draw[color = magenta] ( 0.80000000000000004 , 3. ) -- ( 1. , 3.1666666666666665); 
 \draw[color = magenta] ( 0. , 3.1666666666666665 ) -- ( 1. , 4.); 
 \draw[color = magenta] ( 0. , 4. ) -- ( 1. , 4.833333333333333); 
 \draw[color = magenta] ( 0. , 4.833333333333333 ) -- ( 0.20000000000000001 , 5.); 
 \draw[color = magenta] ( 0.20000000000000001 , 0. ) -- ( 1. , 0.66666666666666663); 
 \draw[color = magenta] ( 1. , 0.66666666666666663 ) -- ( 1.3999999999999999 , 1.); 
 \draw[color = magenta] ( 1.3999999999999999 , 0. ) -- ( 2. , 0.5); 
 \draw[color = magenta] ( 0. , 0.5 ) -- ( 0.59999999999999998 , 1.); 
 \draw[color = magenta] ( 0.59999999999999998 , 1. ) -- ( 1. , 1.3333333333333333); 
 \draw[color = magenta] ( 0. , 1.3333333333333333 ) -- ( 0.80000000000000004 , 2.); 
 \draw[color = magenta] ( 0.80000000000000004 , 2. ) -- ( 1. , 2.1666666666666665); 
 \draw[color = magenta] ( 0. , 2.1666666666666665 ) -- ( 1. , 3.); 
 \draw[color = magenta] ( 0. , 3. ) -- ( 1. , 3.8333333333333335); 
 \draw[color = magenta] ( 0. , 3.8333333333333335 ) -- ( 0.20000000000000001 , 4.); 
 \draw[color = magenta] ( 0.20000000000000001 , 4. ) -- ( 1. , 4.666666666666667); 
 \draw[color = magenta] ( 0. , 4.666666666666667 ) -- ( 0.40000000000000002 , 5.); 
 \draw[color = magenta] ( 0.40000000000000002 , 0. ) -- ( 1. , 0.5); 
 \draw[color = magenta] ( 1. , 0.5 ) -- ( 1.6000000000000001 , 1.); 
 \draw[color = magenta] ( 1.6000000000000001 , 0. ) -- ( 2. , 0.33333333333333331); 
 \draw[color = magenta] ( 0. , 0.33333333333333331 ) -- ( 0.80000000000000004 , 1.); 
 \draw[color = magenta] ( 0.80000000000000004 , 1. ) -- ( 1. , 1.1666666666666667); 
 \draw[color = magenta] ( 0. , 1.1666666666666667 ) -- ( 1. , 2.); 
 \draw[color = magenta] ( 0. , 2. ) -- ( 1. , 2.8333333333333335); 
 \draw[color = magenta] ( 0. , 2.8333333333333335 ) -- ( 0.20000000000000001 , 3.); 
 \draw[color = magenta] ( 0.20000000000000001 , 3. ) -- ( 1. , 3.6666666666666665); 
 \draw[color = magenta] ( 0. , 3.6666666666666665 ) -- ( 0.40000000000000002 , 4.); 
 \draw[color = magenta] ( 0.40000000000000002 , 4. ) -- ( 1. , 4.5); 
 \draw[color = magenta] ( 0. , 4.5 ) -- ( 0.59999999999999998 , 5.); 
 \draw[color = magenta] ( 0.59999999999999998 , 0. ) -- ( 1. , 0.33333333333333331); 
 \draw[color = magenta] ( 1. , 0.33333333333333331 ) -- ( 1.8 , 1.); 
 \draw[color = magenta] ( 1.8 , 0. ) -- ( 2. , 0.16666666666666666); 
 \draw[color = magenta] ( 0. , 0.16666666666666666 ) -- ( 1. , 1.); 
 \draw[color = magenta] ( 1. , 0. ) -- ( 2. , 0.83333333333333337); 
 \draw[color = magenta] ( 0. , 0.83333333333333337 ) -- ( 0.20000000000000001 , 1.); 
 \draw[color = magenta] ( 0.20000000000000001 , 1. ) -- ( 1. , 1.6666666666666667); 
 \draw[color = magenta] ( 0. , 1.6666666666666667 ) -- ( 0.40000000000000002 , 2.); 
 \draw[color = magenta] ( 0.40000000000000002 , 2. ) -- ( 1. , 2.5); 
 \draw[color = magenta] ( 0. , 2.5 ) -- ( 0.59999999999999998 , 3.); 
 \draw[color = magenta] ( 0.59999999999999998 , 3. ) -- ( 1. , 3.3333333333333335); 
 \draw[color = magenta] ( 0. , 3.3333333333333335 ) -- ( 0.80000000000000004 , 4.); 
 \draw[color = magenta] ( 0.80000000000000004 , 4. ) -- ( 1. , 4.166666666666667); 
 \draw[color = magenta] ( 0. , 4.166666666666667 ) -- ( 1. , 5.); 
 \draw[color = blue] ( 0. , 1. ) -- ( 1. , 1.8333333333333333); 
 \draw[color = blue] ( 0. , 1.8333333333333333 ) -- ( 0.20000000000000001 , 2.); 
 \draw[color = blue] ( 0.20000000000000001 , 2. ) -- ( 1. , 2.6666666666666665); 
 \draw[color = blue] ( 0. , 2.6666666666666665 ) -- ( 0.40000000000000002 , 3.); 
 \draw[color = blue] ( 0.40000000000000002 , 3. ) -- ( 1. , 3.5); 
 \draw[color = blue] ( 0. , 3.5 ) -- ( 0.59999999999999998 , 4.); 
 \draw[color = blue] ( 0.59999999999999998 , 4. ) -- ( 1. , 4.333333333333333); 
 \draw[color = blue] ( 0. , 4.333333333333333 ) -- ( 0.80000000000000004 , 5.); 
 \draw[color = blue] ( 0.80000000000000004 , 0. ) -- ( 1. , 0.16666666666666666); 
 \draw[color = blue] ( 1. , 0.16666666666666666 ) -- ( 2. , 1.); 
\end{tikzpicture}
\caption{Cylinder decomposition of $L(2, 2n+1)$ with $n = 2$ in  direction (6,5)}
\label{ZylinderdecomTheta2}
\end{figure}
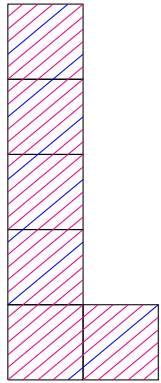

We call
\begin{enumerate}
\item[(i)] $\Psi_g$ the green cylinder. That is the cylinder in direction $(2n+1, 2n+3 )$ with the green saddle connection as the upper boundary.
\item[(ii)] $\Psi_r$ the red cylinder. That is the cylinder in direction $(2n+1, 2n+3 )$ with the red saddle connections as the upper boundary. 
\item[(iii)] $\Theta_m$  the magenta cylinder. That is the cylinder in direction $(2n+2, 2n+1)$ with the magenta saddle connections as the upper boundary. 
\item[(iv)] $\Theta_b$ the blue cylinder. That is the cylinder in direction $(2n+2, 2n+1)$ with the blue saddle connection as the upper boundary. 
\end{enumerate}
These cylinders have the combinatorial length $f_{\Psi_r} = 2n, f_{\Psi_g} = 1, f_{\Theta_m} = 2n+1, f_{\Theta_b} = 1.$ In table \labelcref{intersectionNumberOdd} we computed all necessary intersection numbers. Note that $c_{\Psi_r}= c_{\Theta_m} = c_{\Theta_b} = 1$ and $c_{\Psi_g} = 2.$

It holds
\begin{align*}
\Theta_m &= (2n+1)X_2 + (2n+1)2nX_1 + 2nY_2 + (2n+1)Y_1\\
\Theta_b &= X_2 + 2nX_1 + Y_2\\
\Psi_r &= (2n-1)X_2 + (2(n-1)(2n+1) + 4) X_1 + (2n+1)Y_2 + (2n-1)Y_1\\
\Psi_g &= X_2 + (2n-1)X_1 + Y_2 + 2Y_1
\end{align*}
and
\begin{align*}
D_\Psi( X ) &= X + c_{\Psi_r}f_{\Psi_g}\Omega(X, \Psi_r)\Psi_r
				+ c_{\Psi_g}f_{\Psi_r}\Omega(X, \Psi_g)\Psi_g\\
			&= X -2 \Psi_r + 2\cdot2n\Psi_g\\
			&= X + 4nX_2 + 4n(2n-1)X_1 + 4nY_2 + 8nY_1 \\
			&\quad -2(2n-1)X_2 -2(2(n-1)(2n+1)+4)X_1 - 2(2n+1)Y_2- 2(2n-1)Y_1\\
			&= X + 2X_2 -4X_1 -2Y_2 + (4n+2)Y_1 = 3X -2Y\\
D_\Psi( Y ) &= Y + c_{\Psi_r}f_{\Psi_g}\Omega(Y, \Psi_r)\Psi_r
				+ c_{\Psi_g}f_{\Psi_r}\Omega(Y, \Psi_g)\Psi_g\\	
			&= Y - 2\Psi_r + 4n\Psi_g\\
			&= 2X- Y	
\end{align*}
and
\begin{align*}
D_\Theta(X) & =  X + c_{\Theta_b}f_{\Theta_m}\Omega(X, \Theta_b)\Theta_b
				+ c_{\Theta_m}f_{\Theta_b}\Omega(X, \Theta_m)\Theta_m\\
			&= X - (2n+1)\Theta_b + \Theta_m\\
			&= X - (2n+1)X_2 - (2n+1)2nX_1 - (2n+1)Y_2\\
			& \quad + (2n+1)X_2 + (2n+1)2nX_1 + 2nY_2 + (2n+1)Y_1\\
			&= X - Y_2 + (2n+1)Y_1 = X-Y\\
D_\Theta(X) & =  Y + c_{\Theta_b}f_{\Theta_m}\Omega(Y, \Theta_b)\Theta_b
				+ c_{\Theta_m}f_{\Theta_b}\Omega(Y, \Theta_m)\Theta_m= Y			
\end{align*}
So we have 
\begin{align*}
D_\Psi = \begin{pmatrix}
3 & 2\\-2 & -1
\end{pmatrix}
\text{  and  }
D_\Theta =
\begin{pmatrix}
1 & 0 \\ -1 & 1
\end{pmatrix}
\end{align*}
\newpage
\begin{figure}
\centering
\begin{minipage}{.5\linewidth}
\begin{tabular}{c | c | c } 
$\Omega(X_i, \Theta_j)$ & $\Theta_m$ & $ \Theta_b$\\
\hline
$X_2$ & $4n+1$ &  $1$\\
\hline
$X_1$ & $2n$ &  $1$\\
\hline
$X$ & $1$ &  $-1$\\
\end{tabular}

\begin{tabular}{c | c | c } 
$\Omega(X_i, \Psi_j)$ & $\Psi_r$ & $ \Psi_g$\\
\hline
$X_2$ & $4n$ &  $3$\\
\hline
$X_1$ & $2n+1$ &  $1$\\
\hline
$X$ & $-2$ &  $1$\\
\end{tabular}
\end{minipage}
\qquad
\begin{minipage}{.5\linewidth}
\begin{tabular}{c | c | c } 
$\Omega(Y_i, \Theta_j)$ & $\Theta_m$ & $ \Theta_b$\\
\hline
$Y_1$ & $2n+1$ &  $1$\\
\hline
$Y_2$ & $(2n+1)^2$ &  $2n+1$\\
\hline
$Y$ & $0$ &  $0$\\
\end{tabular}

\begin{tabular}{c | c | c } 
$\Omega(Y_i, \Psi_j)$ & $\Psi_r$ & $\Psi_g$\\
\hline
$Y_1$ & $2n-1$ &  $1$\\
\hline
$Y_2$ & $(2n-1)(2n+1) + 2$ &  $2n$\\
\hline
$Y$ & $-2$ &  $1$\\
\end{tabular}
\end{minipage}
\caption{Intersection numbers of cylinder middles and the homology}
\label{intersectionNumberOdd}
\end{figure}

\section{Intersection number and combinatorial length}\label{comp}
Let $\mathcal{O} = \text{L-Origami}(2, 2n) = O((1\dots 2n), (1 \textbf{ } (2n+1)))$ be an origami and $\pi\colon \mathcal{O} \to E$ the corresponding covering of the standard torus $E = \mathbb{C} / \mathbb{Z}^2$ (see e.g. Figure \labelcref{ZylinderDecom2}). Let $\sigma$ be the cycle $(2\dots 2n \textbf{ } 1 \textbf{ } (2n+1))$. In order to compute the intersection numbers, we introduce some special lattice points. This was inspired by \cite{MR2753950}.
\begin{defn}
\begin{enumerate}
\item An   $(n, n+1)$\textbf{-lattice point} is a point $x \in \mathcal{O}$, such that $\pi(x) =(\frac{a}{n+1}, \frac{b}{n}) $ with $a, b \in \mathbb{Z}$.
\item A \textbf{horizontal}  $(n, n+1)$\textbf{-lattice point} is a point $x \in \mathcal{O}$, such that $\pi(x) =(\frac{a}{n+1}, 0) $ with $a \in \mathbb{Z}$.
\item A \textbf{vertical}  $(n, n+1)$\textbf{-lattice point} is a point $x \in \mathcal{O}$, such that $\pi(x) =(0, \frac{a}{n}) $ with $a \in \mathbb{Z}$.
\end{enumerate}
\end{defn}
We notice that the geodesic line in direction $(n, n+1)$ through an  $(n, n+1)$-lattice point meets again an  $(n, n+1)$-lattice point, whence it meets the horizontal edge of a square.
\begin{lem}\label{geoverlauf}
Let $x$ be a horizontal $(n, n+1)$-lattice point at the lower edge of square $i$, which is no singularity and let $\gamma$ be the geodesic line through $x$ in direction $(n, n+1).$ 
\begin{enumerate}
\item[(a)] Let $y$ be the point of $\gamma,$ where it meets the lower edge of a square the next time, in which $y$ is no singularity. Then $y$ lies at the lower edge of square $\sigma(i).$
\item[(b)] If $\pi(x) = (\frac{a}{n+1}, 0 )$, then $\gamma$ meets the point $y$ with $\pi(y) = (\frac{a}{n+1}, 0 )$ and $y$ lies at the upper edge of square $\sigma^{(n+1)}(i),$ if it meets no singularity.
\end{enumerate}
 
\end{lem}
\begin{proof}
\begin{enumerate}
\item[(a)]  Let $2\leq i \leq 2n.$ Then $\gamma$ leaves square $i$, when it reaches the upper edge. This is the lower  edge of square $\sigma(i).$ If $i=1$ then the first coordinate of $\pi(x)> \frac{1}{n+1},$ since $x$ and $y$ are no singularities. Then $\gamma$ leaves square 1 over the right edge and reaches the lower edge of square $2n+1$ = $\sigma(1).$ The case $i =2n+1$ is treated similarly.
\item[(b)] If  $\pi(x) = (\frac{a}{n+1}, 0)$ then $\pi(\gamma(t) = (\frac{a-1}{n+1} \mod 1, 0)$, when $\gamma$ meets the edge of a square the next time. Then we can apply part (a) $n+1$ times. 
\end{enumerate}

\end{proof}

\begin{lem}\label{latpoints}
Each geodesic line $\gamma$ through an $(n, n+1)$-lattice point $x$ in direction $(n, n+1)$ defines a saddle connection.
\end{lem}
\begin{proof}
Let $x = \gamma(0)$ be no singularity. We can assume, that $x$ is a horizontal $(n, n+1)$-lattice point at the lower edge of some square $i$, because $\gamma$ meets one, the next time it crosses an edge of a square. If $\pi(x) = (\frac{a}{n+1}, 0)$ then then $\pi(\gamma(t)) = (\frac{a-1}{n+1} \mod 1, 0)$, when $\gamma$ meets the edge of a square the next time. So we can assume, that $x$ lies over $(0,0).$ 

 The numbers $n+1$ and $2n + 1 = n + (n+1)$ are coprime. (A common  divisor of $n+1$ and $2n +1$ would be a common divisor of $n+1$ and $n = (2n +1 - (n+1))$.) So we get integers $a, b \in \mathbb{Z}$ with $a (n+1) + b(2n+1) = 1.$ So by \Cref{geoverlauf} either $\gamma$ meets the lower left point of square $\sigma^{a (n+1)}(i) = \sigma(i)$ or it meets a singularity before. So $\gamma$ meets a singularity at least, when the lower left point of square $\sigma^k(i)$ is a singularity.
\end{proof}

Since $\mathcal{O}$ has one singularity of degree 3, there are 3 saddle connections in direction $(n, n+1)$. 
\begin{enumerate}
\item $r_1$ starts in the lower left vertex of square 1 and ends in the upper right vertex of $2n + 1.$
\item $g$ starts in the lower left vertex of square $2n+1$ and ends in the upper right vertex of $1.$
\item Hence the last saddle connection $r_2$ starts at the lower left vertex of square 2 and ends at the upper right vertex of square $2n.$
\end{enumerate}

We will now show, that  the red saddle connections $r_1$ and $r_2$ are the upper boundary of a cylinder, namely $\Theta_r$ and the green saddle connection $g$ is the upper boundary of a cylinder, namely $\Theta_g.$ For this we use separatrix diagrams. There is a nice introduction to separatrix diagrams in \cite{MR2000471}. We will use just the ribbon graph structure of the separatrix diagram (without the pairings of the boundary components). The separatrix diagram of the origami $L(2, 2n)$ with saddle connections in direction $(n, n+1)$ is shown in Figure \labelcref{separatrix}. The cyclic order of the vertex is as drawn. 

\begin{figure}
\begin{center}
\begin{tikzpicture}
\draw[->, color = red] (0,0) -- (0.58 , 1.15);
\draw[ color = red] (0,0) -- (-0.58 , 1.15);
\draw[->, color = red] (0.58 , 1.15) to [out= 54.7, in = 125.3] (-0.58 , 1.15);
\node at (0,1)[above] {$r_1$};
\node at (0,2)[above] {$r_2$}; 
\node at (0,-1)[below] {$g$}; 
\draw[ color = red] (0,0) -- (2,0);
\draw[<-, color = red] (-2,0) -- (0,0);
\draw[<-, color = red] (2,0) to [out = 0, in = 0] (0,2);
\draw[ color = red] (0,2) to [out = 180, in = 180](-2,0); 
\draw[->, color = green] (0,0) -- (0.58 , -1.15);
\draw[ color = green] (0,0) -- (-0.58 , -1.15);
\draw[->, color = green] (0.58 , -1.15) to [out= 305.3, in = 234.7] (-0.58 , -1.15); 
\draw[fill] (0,0) circle (1pt);
\end{tikzpicture}
\end{center}
\caption{The separatrix diagram of L(2, 2n+1) for saddle connections in direction (n, n+1)}
\label{separatrix}
\end{figure}
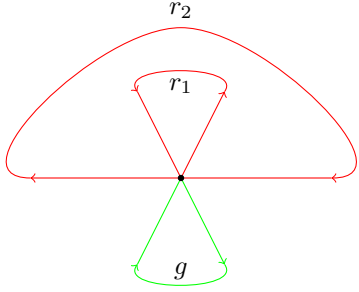

We can find the cylinders with $r_1$ as a part of the upper boundary as follows: We follow the edge $r_1$ until we reach a vertex again. Then we follow the the next edge in the reserved cyclic order of the vertex. This is the edge $r_2$. Then we follow the next edge in the reserved cyclic order, which is again $r_1$ with which we started. So we have found the upper boundary of a cylinder.
With the same procedure we see that $g$ is the upper boundary of a cylinder.

We show now, why this procedure gives the boundary's of the cylinders. We follow a saddle connection until we reach a singularity (which corresponds to a vertex in the separatrix diagram). If we want to know the boundary of the cylinder adjacent to our starting saddle connection, we follow a small path from the saddle connection anti clockwise around the singularity, until we reach a saddle connection again. In the separatrix diagram this is the next edge in the reserved cyclic order. We found the entire upper boundary of the cylinder, when we reach the starting saddle connection again. 

Let us count the intersection numbers of the saddle connections and our chosen basis of the homology states in Figure \labelcref{intersectionNumberEven}. We can represent an element of our basis of the homology by a horizontal or vertical curve $c$ through $(n, n+1)$-lattice points. Then the saddle connection meets $c$ exactly at the $(n, n+1)$-lattice points. 

We will first compute the intersection numbers with  $g$ and conclude the intersection numbers with $\Theta_r$ from this. 

Let us compute the intersection number $\Omega(Y_2, \Theta_g)$. We represent $Y_2$ by the left border of the origami. The saddle connection $g$ runs trough square $2n+1$ to square 1 at the point that lies over $(0, \frac{1}{n})$ on the right edge of square 1 and then up  through the squares $2,\dots, n$, while it meets the left border each of these square once. For sure $n$ this is the left upper vertex which is equal to the right upper vertex. Next $\Theta_g$ runs through square $n+1$ without hitting its vertical edge. Finally $\Theta_g$ runs through squares $n+2, \dots, 2n$, hitting the left edge of each of them once, until it reaches square 1, where it runs in a singularity. During this $\Theta_g$ meets the left border $n-1$ times. In total we have $\Omega(Y_2, \Theta_g) = -(n+ (n-1)) = -(2n-1).$

The other intersection numbers with $\Theta_g$ can be computed similar.

By  \Cref{latpoints}, each   $(n, n+1)$-lattice point lies on a saddle connection. So any lattice point, which does not lie on the green saddle connection, meets a red one. So we have
\begin{align*}
\Omega(X_1, r) = (n + 1) - 1 = n, 
\end{align*}
because $X_1$ contains $n+1$ lattice points, from which 1 is green.
Analogously 
\begin{align*}
\Omega(X_2, r) &= 2(n + 1) - 3 = 2n -1\\
\Omega(Y_1, r) &= n - 1\\
\Omega(Y_2, r) &= 2nn - (2n-1) = 2n(n-1)+1\\
\end{align*}
Finally, we determine the combinatorial length of both cylinders.
\begin{lem} \begin{enumerate}
\item[(a)] The combinatorial length of $g$ is 2.
\item[(b)] The combinatorial length of $r$ is $n-1.$
\end{enumerate} 
\end{lem}
\begin{proof}
We note that the combinatorial length of a curve $\gamma\colon [0,1]\to X$ of a covering $\pi \colon X \to Y$ is the multiplicity of the curve $\pi(\gamma).$ Since the geodesic line is determined by the direction and a point of the geodesic line, the multiplicity of the curve $\gamma$ is the number  $\#\{t \in (0,1], \pi(\gamma(t)) = \pi(\gamma(0))\}.$
\begin{enumerate}
\item[(a)] The green saddle connection meets a point $x$ with $\pi(x) = (0,0)$ at the upper right vertex of square $n$ and at the upper  right vertex of the square 1. Hence the combinatorial length is 2.
\item[(b)] By Lemma \labelcref{latpoints} the red saddle connection meets the upper right vertex of each square, which meets not the green saddle connection. Hence the combinatorial length is $2n +1 - 2 = 2n -1.$
\end{enumerate}
\end{proof}

\bibliography{lorigamis}
\bibliographystyle{unsrtdin}


\end{document}